\documentclass[11pt]{amsart}

\makeindex

\usepackage{amsfonts,graphics,amsmath,amsthm,amsfonts,amscd,amssymb,amsmath,latexsym,multicol,euscript}
\usepackage{epsfig,url}
\usepackage{flafter}
\usepackage{hyperref}
\usepackage[all,cmtip,line]{xy}
\usepackage{array}
\usepackage[english]{babel}
\usepackage{overpic}
\usepackage{subfig}
\usepackage{multirow}

\usepackage{pgf,tikz}
\usepackage{mathrsfs}
\usetikzlibrary{arrows}

\usepackage{wrapfig}

\usepackage{enumitem}

\theoremstyle{definition}

\theoremstyle{theorem}

\newtheorem{theorem}[equation]{Theorem}
\newtheorem{main theorem}[equation]{Main Theorem}

\newtheorem{corollary}[equation]{Corollary}

\newtheorem{lemma}[equation]{Lemma}

\newtheorem{theorem*}{Theorem}
\newtheorem{corollary*}[theorem*]{Corollary}
\newtheorem{conjecture*}[theorem*]{Conjecture}

\theoremstyle{definition}
\newtheorem{example}[equation]{Example}
\newtheorem{definition}[equation]{Definition}
\newtheorem{definition-lemma}[equation]{Definition-Lemma}

\newtheorem*{definition*}{Definition}

\newtheorem{remark}[equation]{Remark}
\numberwithin{equation}{section}

\newcommand\R{\mathbb{R}}
\newcommand\Q{\mathbb{Q}}
\newcommand\Z{\mathbb{Z}}
\newcommand\C{\mathbb{C}}
\renewcommand\P{\mathbb{P}}

\newcommand\eps{\varepsilon}
\renewcommand\epsilon{\varepsilon}

\newcommand{\mc}{\mathcal}

\DeclareMathOperator{\codim}{codim}

\DeclareMathOperator{\im}{Im}
\DeclareMathOperator{\ord}{ord}

\DeclareMathOperator{\vol}{vol}

\DeclareMathOperator{\var}{Var}
\DeclareMathOperator{\rank}{rank}
\DeclareMathOperator{\Spec}{Spec}

\newcommand{\bm}{\mathbf B_-}  

\newcommand{\bp}{\mathbf B_+}  
\newcommand{\okbd}{\Delta}

\begin{document}

\title[A product formula for volumes of divisors via Okounkov bodies]{A product formula for volumes of divisors \\via Okounkov bodies}

\author{Sung Rak Choi} \address{Department of Mathematics, Yonsei University, Seoul, Korea} \email{sungrakc@yonsei.ac.kr}

\author{Seung-Jo Jung} \address{Department of Mathematical Sciences, Seoul National University, Seoul, Korea} \email{toyul419@snu.ac.kr}

\author{Jinhyung Park} \address{Department of Mathematics, Sogang University, Seoul, Korea} \email{parkjh13@sogang.ac.kr}

\author{Joonyeong Won} \address{School of Mathematics, Korea Institute for Advanced Study, Seoul, Korea} \email{leonwon@kias.re.kr}

\date{\today}
\thanks{S. Choi and J. Park were partially supported by NRF-2016R1C1B2011446. S.-J. Jung was partially supported by NRF-2017R1C1B1005166.}

\begin{abstract}
We generalize Kawamata's product formula for volumes of canonical divisors to arbitrary divisors using Okounkov bodies.
\end{abstract}

\maketitle


\section{Introduction}
The \emph{volume} of a $\Q$-divisor $D$ on an $n$-dimensional smooth projective variety $X$ is defined as
$$
\vol_X(D):= \limsup_{m \to \infty} \frac{h^0(X, \mathcal{O}_X(\lfloor mD \rfloor))}{m^{n}/n !}
$$
and the \emph{restricted volume} of $D$ on $X$ along a subvariety $V \subseteq X$ is defined as
$$
\vol_{X|V}(D):=\limsup_{m \to \infty} \frac{\dim \im\Big[H^0(X, \mathcal{O}_X(\lfloor mD \rfloor)) \to H^0(V, \mathcal{O}_X(\lfloor mD \rfloor)|_V)\Big]}{m^{\dim V}/\dim V!}.
$$
When $V \not\subseteq \bm(D)$, the \emph{augmented restricted volume} of $D$ along a subvariety $V \subseteq X$ is defined as
$$
\vol_{X|V}^+(D):=\lim_{\eps \to 0} \vol_{X|V}(D+\eps A)
$$
where $A$ is an ample $\Q$-divisor on $X$. It is independent of the choice of $A$.
We refer to~\cite{CHPW-okbd I} and~\cite{elmnp-restricted vol and base loci} for more details.

The main result of this article is the following.

\begin{theorem}\label{main1}
Let $f \colon X \to Y$ be a surjective morphism with connected fibers between smooth projective varieties and $F$ a general fiber of $f$.
Consider a big divisor $D$ on $X$ such that $D \sim_{\Q} f^*D_Y + R$ for some big divisor $D_Y$ on $Y$ and an effective divisor $R$ on $X$. Then we have
$$
\frac{\vol_X(D)}{\dim X !} \geq \frac{\vol_Y (D_Y)}{\dim Y !} \cdot \frac{\vol_{X|F}^+(R)}{\dim F!}.
$$
Suppose furthermore that the coherent sheaf $f_*\mathcal{O}_X(mR)$ is weakly positive for every sufficiently divisible integer $m>0$. Then we have
$$
\frac{\vol_X(D)}{\dim X !} \geq \frac{\vol_Y (D_Y)}{\dim Y !} \cdot \frac{\vol_{F}(R|_F)}{\dim F!}.
$$
\end{theorem}
Let us recall the definition of weak positivity, which is the key condition in Theorem \ref{main1}.

\begin{definition*}[{\cite[Definition 1.2]{V}}]
Let $\mathcal{F}$ be a coherent torsion-free sheaf on a smooth projective variety $X$.
For an open subvariety $U \subseteq X$, we say that $\mathcal{F}$ is \emph{weakly positive over $U$} if for every ample Cartier divisor $H$ on $X$ and every positive integer $m>0$, there exists a positive integer $k>0$ such that
\[
H^0(X,\hat{S}^{mk}(\mathcal{F})\otimes \mathcal{O}_X(kH))\otimes \mathcal{O}_X\to \hat{S}^{mk}(\mathcal{F})\otimes \mathcal{O}_X(kH)
\]
is surjective at each point in $U$, where $\hat{S}^{mk}(\mathcal{F}):=\big(S^{mk}(\mathcal{F})\big)^{\vee\vee}$ is the double dual of the sheaf $S^{mk}(\mathcal{F})$.
We say that $\mathcal{F}$ is \emph{weakly positive} if there exists an open subvariety $U$ such that $\mathcal{F}$ is weakly positive over $U$.
\end{definition*}

\medskip

The last inequality in Theorem \ref{main1} improves the first one since $\vol_F(R|_F)\geq\vol_{X|F}^+(R)$ holds in general. Note also that by letting $D=K_X$ in the second inequality of Theorem \ref{main1}, we recover Kawamata's theorem~\cite[Theorem 0.1]{K}, thereby obtaining a generalization.
Recall that $f_* \omega_{X/Y}^{\otimes m}$ is weakly positive for every integer $m>0$ by Viehweg \cite[Theorem III]{V}.

\begin{corollary}\label{kawamata}
Let $f \colon X \to Y$ be a surjective morphism with connected fibers between smooth projective varieties and $F$ a general fiber of $f$.
Then we have
$$
\frac{\vol_X(K_X)}{\dim X !} \geq \frac{\vol_Y (K_Y)}{\dim Y !} \cdot \frac{\vol_{F}(K_F)}{\dim F!}.
$$
\end{corollary}
\smallskip

Note that this inequality holds trivially if $Y$ or $F$ is not of general type.

Recall the conjecture on the subadditivity property for logarithmic Kodaira dimension considered in
\cite[Conjecture 1.2]{Fujino}: for an algebraic fiber space $f \colon X \to Y$ with a sufficiently general fiber $F$, if $D_X$ on $X$ and $D_Y$ on $Y$ are simple normal crossing effective divisors satisfying $\text{Supp} f^*D_Y \subseteq \text{Supp} D_X$, then
$$
\kappa(K_X+D_X) \geq \kappa(K_Y + D_Y) + \kappa(K_F + D_X|_F).
$$
It is tempting to hope that under the same assumptions, Corollary \ref{kawamata} holds, i.e.,
$$
\frac{\vol_X (K_X+D_X)}{\dim X!}\geq\frac{\vol_Y (K_Y+D_Y)}{\dim Y!}\cdot\frac{\vol_F (K_F + D_X|_F)}{\dim F!}.
$$
However, this expectation is far from the truth. Example~\ref{ex0} shows that this expectation is not true even if $(X, D_X)$ and $(Y, D_Y)$ are klt pairs of log general type. See also Remark \ref{Iitaka} for more discussion.

Ever since the introduction of Okounkov bodies in \cite{O1} and \cite{O2} and the pioneering work by  Lazarsfeld--Musta\c{t}\u{a}~\cite{lm-nobody} and Kaveh--Khovanskii~\cite{KK}, a remarkable progress has been made.
Due to the recent results, Okounkov bodies provide a rich and useful point of view on the positivity of divisors (see e.g., \cite{CHPW-asyba}, \cite{KL1}, \cite{KL2}, \cite{Roe}).
The \emph{Okounkov body} $\okbd_{X_\bullet}(D)$ associated to a divisor $D$ on a smooth projective variety $X$ is a closed and bounded convex subset of the Euclidean space $\R^{n}$ where $n=\dim X$. The delicate dependence of $\okbd_{X_\bullet}(D)$ on the admissible flags $X_\bullet$ makes the theory more interesting and deeper. By definition, an admissible flag $X_\bullet$ on $X$ is a sequence of subvarieties
$$
X_\bullet : X=X_0 \supseteq X_1 \supseteq \cdots \supseteq X_{n - 1} \supseteq X_{n}=\{x\}
$$
where each $X_i$ is an irreducible subvariety of codimension $i$ in $X$ that is smooth at the point $x$. See Section \ref{prelimsec} for the precise construction of $\okbd_{X_\bullet}(D)$.

As already noted in~\cite{lm-nobody} and many other recent results, the Okounkov body renders many of the basic properties of asymptotic invariants of divisors in a transparent manner.
One of the first fundamental results in~\cite{lm-nobody} and \cite{KK} states that if $D$ is a big divisor, then for any admissible flag $X_\bullet$ on $X$, we have
\begin{equation}\label{eq.volume formula}
\vol_{\R^{n}}(\okbd_{X_\bullet}(D)) = \frac{1}{n!} \vol_X(D)
\end{equation}
where $\vol_{\R^{n}}$ denotes the Euclidean volume.
This result has been extended to pseudoeffective divisors in~\cite{CHPW-okbd I}.
The results of~\cite{lm-nobody} and \cite{KK} have also inspired a number of other extensions which allowed us to understand various positivity properties of divisors in terms of Okounkov bodies.
By~(\ref{eq.volume formula}), the second inequality in Theorem~\ref{main1} can be interpreted as
$$
\begin{array}{c}
\vol_{\R^{\dim X}}(\okbd_{X_\bullet}(D))\geq \vol_{\R^{\dim Y}}\okbd_{Y_\bullet}(D_Y)\cdot\vol_{\R^{\dim F}}\okbd_{F_\bullet}(R|_F)
\end{array}
$$
for any admissible flags $X_\bullet, Y_\bullet, F_\bullet$ on $X, Y, F$, respectively.

The proof of Theorem~\ref{main1} relies on the analysis of the Okounkov bodies with respect to carefully chosen admissible flags.
Let $Y_\bullet$ be an admissible flag on $Y$ such that $Y_{\dim Y}$ is a general point of $Y$ and $F_{\bullet}$ an admissible flag on the general fiber $F=f^{-1}(Y_{\dim Y})$. Then we can define a \emph{fiber-type admissible flag $X_\bullet$ associated to $Y_\bullet$ and $ F_\bullet$} as an admissible flag on $X$ such that
\[
X_i:=
\begin{cases}
f^{-1}(Y_i) &\text{if $0 \leq i\leq \dim Y$,}\\
F_{i-\dim Y} &\text{if $\dim Y < i\leq \dim X$.}\\
\end{cases}
\]
Since $F$ is a general fiber, it is easy to check that $X_\bullet$ is an admissible flag on $X$.
For $\eps > 0$, let $W_\bullet^{\eps}=W_\bullet(R+\eps A|F)$ be the graded complete linear series of $R+\eps A$ on $X$ restricted to $F$ where $A$ is an ample divisor on $X$.
More precisely, $W_0=\C$ and each $W_m^{\eps}$ for $m\geq 1$ is given by
$$
W_m^{\eps}:=W_m(R+\eps A|F)= \im\big[H^0\big(X, \mathcal{O}_X(\lfloor m(R+\eps A) \rfloor)\big) \to H^0\big(F, \mathcal{O}_X(\lfloor m(R+\eps A) \rfloor)|_F\big)\big].
$$
First, we will observe that
$$
\okbd_{X_\bullet}(D) \supseteq \okbd_{Y_\bullet}(D_Y) \times \bigcap_{\eps > 0} \okbd_{F_\bullet}(W_\bullet^{\eps})
$$
where $\okbd_{Y_\bullet}(D_Y)$ is defined in $\R^{\dim Y}$ and $\okbd_{F_\bullet}(W_\bullet^{\eps})$ is defined in $\R^{\dim F}$ for the decomposition $\R^{\dim X}=\R^{\dim Y}\times\R^{\dim F}$.
We also remark that once we identify $\okbd_{Y_\bullet}(D_Y)$ and $\okbd_{F_\bullet}(W_\bullet^{\eps})$
with $\okbd_{Y_\bullet}(D_Y)\times\{0\}^{\dim F}$ and $\{0\}^{\dim Y}\times\okbd_{F_\bullet}(W_\bullet^{\eps})$, respectively, we can express
$$\okbd_{Y_\bullet}(D_Y)\times\okbd_{F_\bullet}(W_\bullet^{\eps})=\okbd_{Y_\bullet}(D_Y)\times\{0\}^{\dim F}+\{0\}^{\dim Y}\times\okbd_{F_\bullet}(W_\bullet^{\eps})$$
using the Minkowski sum $`+$'.
Note that the above inclusion is an equality  precisely if the Okounkov body $\okbd_{X_\bullet}(D)$  has the shape of a \emph{box}
$$
\okbd_{Y_\bullet}(D_Y) \times \bigcap_{\eps > 0}\okbd_{F_\bullet}(W_\bullet^{\eps}) \subseteq \R^{\dim Y} \times \R^{\dim F} = \R^{\dim X}
$$
having the \emph{base} $\okbd_{Y_\bullet}(D_Y) \subseteq \R^{\dim Y}$ and the \emph{height} $\bigcap_{\eps > 0}\okbd_{F_\bullet}(W_\bullet^{\eps}) \subseteq \R^{\dim F}$.
If $f_*\mathcal{O}_X(mR)$ is weakly positive for every sufficiently divisible integer $m>0$, then  by Lemma~\ref{weakpos}, we have $\bigcap_{\eps > 0}\okbd_{F_\bullet}(W_\bullet^{\eps})=\okbd_{F_\bullet}(R|_F)$. Now, Theorem~\ref{main1} follows from~(\ref{eq.volume formula}). 
See Section \ref{pfsec} for details.

We are now interested in when the equality holds in Corollary~\ref{kawamata}.
The following result confirms Kawamata's expectation in~\cite[Remark 0.2]{K}.

\begin{theorem}\label{main2}
Let $f \colon X \to Y$ be a surjective morphism with connected fibers between smooth projective varieties and $F$ a general fiber of $f$. Suppose that both $Y$ and $F$ are varieties of general type. Then the following are equivalent:\\[-10pt]
$
\begin{array}{ll}
 \text{\begin{minipage}[t]{0,3\textwidth}$$(1)\;\frac{\vol_X(K_X)}{\dim X !} = \frac{\vol_Y (K_Y)}{\dim Y !} \cdot \frac{\vol_{F}(K_F)}{\dim F!}.$$\end{minipage}}& \\ \\[-8pt]
  $$(2)\;\text{$\okbd_{X_\bullet}(K_X)\cong\okbd_{Y_\bullet}(K_Y)\times\okbd_{F_\bullet}(K_F)$ for any fiber-type admissible flag $X_\bullet$ on $X$ associated to}$$& \\
  \phantom{(2)}\;\text{admissible flags $Y_\bullet,F_\bullet$ on $Y, F$, respectively.}$$& \\  \\[-10pt]
  $$(3)\;\text{$f$ is birationally isotrivial.}$$&
\end{array}
$
\end{theorem}

In the proof of Theorem \ref{main2}, we use the properties of Okounkov bodies. See Section \ref{pfsec} for details. The equivalence $(1) \Leftrightarrow (3)$ of Theorem \ref{main2} was also announced by Tsuji in~\cite[Corollary 4.7]{T}.

Example~\ref{ex1} shows that the inequality of Theorem \ref{main1} can be strict for a birationally isotrivial fibration.
On the other hand, Example \ref{ex2} shows that the inequality of Theorem \ref{main1} can be an equality for a non-birationally isotrivial fibration. From these observations, one may say that the birational isotriviality of a fibration is governed by the property of the canonical divisors, not other divisors.
These examples also show that a log version of Theorem~\ref{main2} does not hold even for klt pairs. See Section \ref{exsec} for details.

\medskip

The rest of the paper is organized as follows. We begin with Section \ref{prelimsec} to recall some basic notions.
Section~\ref{pfsec} is devoted to proving our main results, Theorem~\ref{main1}, Corollary~\ref{kawamata}, and Theorem~\ref{main2}. In Section~\ref{exsec}, we present relevant examples.

\smallskip

\subsection*{Acknowledgement}
We would like to thank Constantin Shramov for informing us of the paper~\cite{K}.
We are grateful to the referees for their careful reading of the paper and numerous suggestions for improvement.

\section{Preliminaries}\label{prelimsec}
Throughout the paper, we work over the field $\C$ of complex numbers. A divisor always means a $\Q$-Cartier $\Q$-divisor unless otherwise stated.

In this section, we recall the relevant basic notions.
Let $X$ be an $n$-dimensional projective variety. An \emph{admissible flag} $X_\bullet$ on $X$ is a sequence of irreducible subvarieties of $X$
\[
X_\bullet : X=X_0 \supseteq X_1 \supseteq \cdots \supseteq X_{n - 1} \supseteq X_{n}
\]
where each $X_i$ is of codimension $i$ in $X$ and smooth at the point $X_{n}$. For a divisor $D$ with $|D|_{\Q}=\{ D' \mid D \sim_{\Q} D' \geq 0 \}\neq \emptyset$, we define a valuation-like function
\[
\nu_{X_{\bullet}}\colon |D|_{\Q}\rightarrow \R^n_{\geq 0}
\]
associated to $X_{\bullet}$ as follows. For $D'\in |D|_{\Q}$, we first define
\[
\nu_1:=\nu_1(D'):=\ord_{X_1}(D') \quad \text{and}\quad \nu_2:=\nu_2(D'):=\ord_{X_2}\big((D'-\nu_1(D')X_1)|_{X_1}\big).
\]
For $i>2$, define $\nu_i$ inductively as
\[
\nu_i:=\nu_i(D'):=\ord_{X_i}\big(( \cdots ((D'-\nu_1 X_1)|_{X_1}-\nu_2 X_2 )|_{X_2}-\cdots-\nu_{i-1}X_{i-1})|_{X_{i-1}}\big).
\]
By collecting $\nu_i$'s for $i=1,\ldots, n$, we define $\nu_{X_{\bullet}}(D'):=(\nu_1,\ldots,\nu_n)\in \R^n_{\geq 0}$.

\begin{definition}
Let $D$ be a divisor on $X$ such that $|D|_{\mathbb Q}\neq \emptyset$. The \emph{Okounkov body $\Delta_{X_\bullet}(D)$ associated to the divisor $D$ with respect to an admissible flag $X_{\bullet}$} is defined as the closed and bounded convex subset
\[
\Delta_{X_\bullet}(D):=\text{the closure of the convex hull of $\nu_{X_{\bullet}}(|D|_{\Q})$ in $\R^n_{\geq 0}$}.
\]
\end{definition}

When $D \sim_{\Q} D'$, we have $\okbd_{X_\bullet}(D)=\okbd_{X_\bullet}(D')$ by definition. If $D, D'$ are big and $D$ is numerically equivalent to $D'$, then $\okbd_{X_\bullet}(D)=\okbd_{X_\bullet}(D')$ by \cite[Proposition 4.1]{lm-nobody}.

\begin{remark}\label{additivity}
Let $D, D'$ be divisors on $X$ such that $|D|_{\Q}\neq \emptyset, |D'|_{\Q} \neq \emptyset$. Fix an admissible flag $X_\bullet$ on $X$. By definition, we have $\nu_{X_\bullet}(D_0) + \nu_{X_\bullet}(D_0')=\nu_{X_\bullet}(D_0+D_0')$ for any $D_0 \in |D|_{\Q}, D_0' \in |D'|_{\Q}$. Using this, we can easily check the following property
$$
\okbd_{X_\bullet}(D) + \okbd_{X_\bullet}(D') \subseteq \okbd_{X_\bullet}(D+D').
$$
\end{remark}

Note that our definition of the Okounkov body is equivalent to the one in~\cite{lm-nobody} and \cite{KK} where the above function $\nu_{X_\bullet}$ is applied to the nonzero sections $s$ of each piece $H^0(X,\mc O_X(mD))$ of the graded section ring $\bigoplus_{m \geq 0} H^0(X,\mc O_X(mD))$ and the Okounkov body $\okbd_{X_\bullet}(D)$ is defined as the convex closure of the set of rescaled images $\frac{1}{m}\nu_{X_\bullet}(s)$ in $\R^n$.
This equivalent construction can be generalized to a graded linear (sub)series $W_\bullet$ given by a divisor on $X$ to construct the Okounkov body $\okbd_{X_\bullet}(W_\bullet)$ associated to $W_\bullet$ with respect to an admissible flag $X_\bullet$.
Now, let $W_\bullet:=W_\bullet(D|X_{n-k})$ be a graded linear series given by a divisor $D$ on $X$ restricted to $X_{n-k}$: for each $m>0$, we let
$$W_m=W_m(D|X_{n-k}):= \im\big[H^0\big(X, \mathcal{O}_X(\lfloor mD \rfloor)\big) \to H^0\big(X_{n-k}, \mathcal{O}_X(\lfloor mD \rfloor)|_{X_{n-k}}\big)\big].
$$
We may regard the partial admissible flag
$$
X_{n-k \bullet} : X_{n-k} \supseteq X_{n-k-1} \supseteq \cdots \supseteq X_{n-1}\supseteq X_n=\{x\}
$$
as an admissible flag on $X_{n-k}$ that is a $k$-dimensional projective variety.
We define the Okounkov body which we often regard as a subset of $\R^n_{\geq 0}$
$$
\okbd_{X_{n-k \bullet}}(D):=\okbd_{X_{n-k \bullet}}(W_\bullet) \subseteq \R^{k}_{\geq 0}\cong \{ 0 \}^{n-k} \times \R^{k}_{\geq 0} \subseteq \R^n_{\geq 0}.
$$
For more details, we refer to \cite{lm-nobody}, \cite{KK} and \cite{CHPW-okbd I}.

\begin{definition}
The \emph{Iitaka dimension of a divisor $D$} on $X$ is defined as
\[
\kappa(D):=\max\left\{k\in \Z_{\geq 0}\; \left|\; \limsup_{m\to \infty}\frac{h^0(X,\mathcal{O}_X(\lfloor mD\rfloor))}{m^k}>0\right.\right\}
\]
if $h^0(X,\mathcal{O}_X(\lfloor mD\rfloor))>0$ for some integer $m>0$ and $\kappa(D):=-\infty$ otherwise.
\end{definition}

\smallskip

\begin{lemma}[{\cite[Proposition 1]{Fujita}, \cite[Proposition 1.14]{Mori}}]\label{fujita}
Let $f \colon X \to Y$ be a surjective morphism with connected fibers between smooth projective varieties, $F$ a general fiber of $f$, and $D$ a divisor on $X$. Then there exists a big divisor $D_Y$ on $Y$ such that $\kappa(D - f^*D_Y) \geq 0$ if and only if $\kappa(D)=\dim Y + \kappa(D|_F)$.
\end{lemma}

Let us recall the definition of variation of a fibration.

\begin{definition}[{\cite[p.329]{V}}]
Let $f\colon X \to Y$ be a surjective morphism with connected fibers between smooth projective varieties. Let $\overline{\C(Y)}$ be the algebraic closure of the function field of $Y$. The \emph{variation of $f$}, denoted by $\var(f)$, is defined as the minimum of the transcendental degrees $\text{tr.}\deg_\C L$ of algebraically closed subfields $L\subseteq\overline{\C(Y)}$ such that $F\times_{\Spec (L)} \Spec (\overline{\C(Y)})$ is birational to $X\times_{Y} \Spec (\overline{\C(Y)})$ for some smooth projective variety $F$ over $L$.
\end{definition}

Note that the following statements are equivalent:
\begin{enumerate}
 \item $f$ is birationally isotrivial.
 \item $\var(f)=0$.
 \item there exists a generically finite cover $\tau \colon Y' \to Y$ such that the fiber product $X \times_Y Y'$ is birational to $Y' \times F$, where $F$ is a general fiber of $f$.
\end{enumerate}

\begin{theorem}[{\cite[Theorem 1.1]{K1}, \cite[Theorem 1.20]{V2}}]\label{var}
Let $f \colon X \to Y$ be a surjective morphism with connected fibers between smooth projective varieties, and $F$ a general fiber of $f$. Suppose that $F$ is a variety of general type. If $D_Y$ is any divisor on $Y$ with $\kappa(D_Y) \geq 0$, then we have
$$
\kappa(K_{X/Y} + f^*D_Y ) \geq \kappa(K_F) + \max\{ \kappa(D_Y), \var(f) \}.
$$
\end{theorem}

We recall the following well known result with the sketch of the proof for the readers' convenience.
\begin{lemma}\label{lem-eff div exists}
Let $f\colon X\to Y$ be a surjective morphism with connected fibers between smooth projective varieties. Let $L$ be a Cartier divisor on $X$. Then there exists an effective divisor $B$ on $X$ such that $\codim_Yf(B)\geq2$ and $(f_*\mathcal{O}_X(L))^{\vee\vee}=f_*(\mathcal{O}_X(L+B))$.
\end{lemma}
\begin{proof}
Let $U$ be the maximal open subvariety of $Y$ on which $f_*\mathcal{O}_X(L)$ is locally free and let $V:=f^{-1}(U)$. Since $f_*\mathcal{O}_X(L)$ is torsion-free, we have $\codim_Y(Y\setminus U)\geq2$. Let $D:=X\setminus V$. If $\codim_XD\geq 2$, then put $B=0$ and we are done. Otherwise, $D$ is a divisor on $X$ and we have
$$
\cdots \subseteq f_*\mathcal{O}_X(L+mD)\subseteq f_*\mathcal{O}_X(L+(m+1)D)\subseteq\cdots\subseteq (f_*\mathcal{O}_X(L))^{\vee\vee}.
$$
Since $(f_*\mathcal{O}_X(L))^{\vee\vee}$ is a coherent sheaf, this ascending chain must stabilize at some $m_0>0$, i.e.,
$$
f_*\mathcal{O}_X(L+mD)=f_*\mathcal{O}_X(L+(m+1)D)=\cdots=(f_*\mathcal{O}_X(L))^{\vee\vee}\;\;\text{for all}\;m\geq m_0.
$$
Now by putting $B=m_0D$, we are done.
\end{proof}

\section{Proofs of main results}\label{pfsec}

In this section, we prove the theorems stated in the introduction. The following is one of the main ingredients in the proof of Theorem \ref{main1}.

\begin{lemma}\label{weakpos}
Let $f \colon X \to Y$ be a surjective morphism with connected fibers between smooth projective varieties and $F$ a general fiber of $f$. Consider a big divisor $D$ on $X$ such that $D \sim_{\Q} f^*D_Y + R$ where $D_Y$ is a big divisor on $Y$ and $R$ is an effective divisor on $X$. Suppose that  $f_* \mathcal{O}_X(mR)$ is weakly positive for every sufficiently divisible integer $m>0$. Then we have
$$
 \vol_{X|F}(D)=\vol_{X|F}^+(R)=\vol_F(D|_F)=\vol_F(R|_F).
$$
\end{lemma}

\begin{proof}
Note first that by Lemma \ref{fujita}, $D|_F=R|_F$ is big.
Since we have
$$
\vol_{X|F}^+(R) \leq \vol_F(R|_F)~~\text{and}~~\vol_{X|F}(D) \leq \vol_F(R|_F),
$$
it is enough to show the following inequalities
\begin{equation}\label{2inequalities}\tag{$\#$}
\vol_{X|F}^+(R) \geq \vol_F(R|_F) ~~\text{and}~~ \vol_{X|F}(D) \geq \vol_F(R|_F).
\end{equation}
To prove these inequalities, we claim the following:
\begin{equation}\label{claim*}\tag{$\star$}
\vol_{X|F}\left(R + \frac{1}{m}f^*H \right) \geq \vol_F(R|_F)
\end{equation}
for an ample divisor $H$ on $Y$ and for any sufficiently large integer $m>0$.

Assuming the claim (\ref{claim*}), we first prove the inequalities in (\ref{2inequalities}).
We recall that by \cite[Theorem 5.2]{elmnp-restricted vol and base loci}, $\vol_{X|F} \colon \text{Big}^F(X) \to \R_{\geq 0}$ is a continuous function where $\text{Big}^F(X)$ denotes the set of all $\R$-divisor classes $\xi$ such that $F$ is not properly contained in any irreducible component of $\bp(\xi)$.
By Lemma~\ref{fujita}, $R+ \frac{1}{m}f^*H$ is big for any ample divisor $H$. Note also that since $F$ is a general fiber of $f$, we may assume that $F \not\subseteq \bm(R)\ \text{and}\ F \not\subseteq \bp \left(R + \frac{1}{m}f^*H \right)$ for any $m>0$.
Thus by claim (\ref{claim*}), we obtain the first inequality of (\ref{2inequalities}):
$$
\vol_{X|F}^+(R)=\lim_{m \to \infty} \vol_{X|F}\left(R + \frac{1}{m}f^*H \right) \geq \vol_F(R|_F).
$$
On the other hand, using (\ref{claim*}) we have
$$
\vol_{X|F}\left(D + \frac{1}{m}f^*H \right) \geq \vol_{X|F}\left(R + \frac{1}{m}f^*H \right) \geq \vol_F(R|_F).
$$
Since $F$ is a general fiber, we may assume that
$F \not\subseteq \bp(D)\ \text{and}\ F \not\subseteq \bp \left(D + \frac{1}{m}f^*H \right)$.
Thus we obtain the second inequality of (\ref{2inequalities}):
$$
\vol_{X|F}(D)=\lim_{m \to \infty}\vol_{X|F}\left(D + \frac{1}{m}f^*H \right) \geq \vol_F(R|_F).
$$

Now it remains to prove the claim (\ref{claim*}).
For any $\epsilon >0$, by applying Fujita's approximation \cite[Theorem]{Fujita2} to $R|_F$, we can find a birational morphism $g \colon \widetilde{F} \to F$ and an ample divisor $L'$ on $\widetilde{F}$ such that $g^*(R|_F)-L'$ is effective and $\vol_{\widetilde{F}}(L') > \vol_{F}(R|_F) - \epsilon$. By assumption, $f_*\mathcal{O}_X(m_1R)$ is weakly positive for any sufficiently large and divisible $m_1>0$. Using this condition, we prove
\begin{equation}\label{claim**}\tag{${\star\star}$}
\vol_{X|F}(m_1R+f^*H)\geq\vol_{\widetilde{F}}(m_1L'),
\end{equation}
which will be used to obtain the claim (\ref{claim*}).

By \cite[Lemma 7.3]{V}, we have the following commutative diagram
\[
\begin{split}
\xymatrix{
X' \ar[d]_-{f'} \ar[r]^-{\tau'} & X \ar[d]^-{f} \\
Y' \ar[r]_-{\tau} & Y
}
\end{split}
\]
where $\tau \colon Y' \to Y$ is a birational morphism with $Y'$ smooth, $X'$ is a resolution of singularities of the main component of $X \times_Y Y'$, and $f' \colon X' \to Y'$ is the induced morphism such that every divisor $B'$ on $X'$ with $\codim f'(B') \geq 2$ is contained in the exceptional locus of $\tau' \colon X' \to X$.
Let $R':=\tau'^*R$. We may assume that $\tau$ is an isomorphism over a neighborhood of $f(F)$, so we may consider $F$ also as a general fiber of $f'$ so that $R'|_F=R|_F$.
For a sufficiently large and divisible integer $m_1>0$, the canonical map $\tau^*f_*\mathcal{O}_{X}(m_1R) \to f'_*\mathcal{O}_{X'}(m_1R')$ is surjective over some open subvariety $U$ of $Y'$. Thus $f'_*\mathcal{O}_{X'}(m_1R')$ is weakly positive on $Y'$. Let $H$ be an ample Cartier divisor on $Y$, and $H':=\tau^*H$. Then there is some integer $k>0$ such that $\hat{S}^k f'_*\mathcal{O}_{X'}(m_1R') \otimes \mathcal{O}_{Y'}(kH')$ is generated by global sections over some open subvariety of $Y'$.

For any integer $m_0 > 0$, we consider the natural map
$$
\varphi_{m_0} \colon \hat{S}^{m_0}\big(\hat{S}^k\big(f'_*\mathcal{O}_{X'}(m_1R') \big) \otimes \mathcal{O}_{Y'}(kH')\big) \to \big( f'_*\mathcal{O}_{X'}(m_0m_1 kR') \otimes \mathcal{O}_{Y'}(m_0kH') \big)^{\vee\vee}.
$$
By Lemma \ref{lem-eff div exists}, there exists an effective divisor $B$ on $X'$ such that $\codim f'(B) \geq 2$  and
$$
\big(f'_*\mathcal{O}_{X'}\big(m_0m_1 kR' + f'^*(m_0 k H') \big) \big)^{\vee\vee} = f'_*\mathcal{O}_{X'}\big(m_0m_1 kR' + f'^*(m_0 k H') +B \big).
$$
Note that $B$ is a $\tau'$-exceptional divisor. Thus we have
$$
H^0\big(Y', \big( f'_* \mathcal{O}_{X'}\big(m_0m_1 kR' + f'^*(m_0 k H') \big) \big)^{\vee\vee} \big)= H^0\big(X, \mathcal{O}_X \big( m_0m_1 kR + f^*(m_0kH) \big) \big).
$$
Since $m_1$ is sufficiently large, we may assume that $L:=m_1L'$ is a very ample Cartier divisor.
If $m_0>0$ is sufficiently large, then $S^{m_0}S^k H^0(\widetilde{F}, \mathcal{O}_{\widetilde{F}}(L)) \to H^0(\widetilde{F}, \mathcal{O}_{\widetilde{F}}(m_0 k L))$ is surjective. Consider the following commutative diagram
\[
\begin{split}
\xymatrix{
S^{m_0}S^kH^0(F, \mathcal{O}_F(m_1R|_F)) \ar[r] & H^0(F, \mathcal{O}_F(m_0m_1kR|_F)) \\
S^{m_0}S^kH^0(\widetilde{F}, \mathcal{O}_{\widetilde{F}}(L) )\ar@{^{(}->}[u] \ar@{->>}[r] & H^0(\widetilde{F}, \mathcal{O}_{\widetilde{F}}(m_0kL)).  \ar@{^{(}->}[u]
}
\end{split}
\]
Since $\rank \im( \varphi_{m_0}  ) = \dim \im \big[ S^{m_0}S^kH^0(F, \mathcal{O}_F(m_1R|_F)) \to H^0(F, \mathcal{O}_F(m_0m_1kR|_F))\big]$, we see that
$$
\rank \im( \varphi_{m_0}  )  \geq h^0(\widetilde{F}, \mathcal{O}_{\widetilde{F}}(m_0 kL)).
$$
Now, the generic global generation of $\hat{S}^k(f'_*\mathcal{O}_{X'}(m_1R')) \otimes \mathcal{O}_{Y'}(kH')$ implies the surjectivity of the map $\psi$ in the following commutative diagram
\[
\begin{split}
\xymatrix{
S^{m_0}S^kH^0(F,\mathcal{O}_F(m_1R|_F)) \ar[r] & H^0(F, \mathcal{O}_F(m_0m_1kR|_F)) \\
H^0(\hat{S}^{m_0}(\hat{S}^k(f'_*\mathcal{O}_{X'}(m_1R'))\otimes\mathcal{O}_{Y'}(kH'))\ar@{->>}[u]^{\psi} \ar[r] & H^0(X,\mathcal{O}_X(m_0m_1kR+f^*(m_0kH))).  \ar[u]
}
\end{split}
\]
Hence, we obtain
$$
\dim \im \big[ H^0\big(X, \mathcal{O}_X \big( m_0m_1 kR + f^*(m_0kH) \big) \big) \to H^0\big(F, \mathcal{O}_F(m_0m_1 k R|_F)\big) \big] \geq h^0\big(\widetilde{F}, \mathcal{O}_{\widetilde{F}}(m_0 kL)\big).
$$
This implies $\vol_{X|F}(m_1R+f^*H)\geq\vol_{\widetilde{F}}(L)$ and the inequality (\ref{claim**}) follows since $L=m_1L'$.
Using this inequality, we obtain
$$
 \vol_{X|F}\left( R + \frac{1}{m_1}f^*H \right) \geq \vol_{\widetilde{F}}\left( \frac{1}{m_1}L \right) = \vol_{\widetilde{F}}(L') > \vol_F(R|_F) - \epsilon.
$$
Since $\epsilon$ can be arbitrarily small and $m_1$ can be arbitrarily large, we obtain the claim (\ref{claim*}). This completes the proof.
\end{proof}

We are now ready to prove Theorem \ref{main1}.

\begin{proof}[Proof of Theorem \ref{main1}]
By Lemma \ref{fujita}, $R|_F=D|_F$ is a big divisor on a general fiber $F$ of $f \colon X \to Y$.
Let $Y_\bullet$ be an admissible flag on $Y$ such that $f^{-1}(Y_{\dim Y})=F$, and $F_\bullet$ an admissible flag on $F$. We now consider a fiber-type admissible flag $X_\bullet$ associated to $Y_\bullet$ on $Y$ and $F_\bullet$ on $F$. Recall that this is an admissible flag on $X$ such that
\[
X_i:=
\begin{cases}
f^{-1}(Y_i) &\text{if $0 \leq i\leq \dim Y$,}\\
F_{i-\dim Y} &\text{if $\dim Y < i\leq \dim X$.}\\
\end{cases}
\]
Take an ample divisor $A$ on $X$.
For $\eps > 0$, let $W_\bullet^{\eps}:=W_\bullet(R+\eps A|F)$ be the graded complete linear series of $R+\eps A$ on $X$ restricted to $F$.

Note that $\okbd_{X_\bullet}(f^*D_Y)=\okbd_{Y_\bullet}(D_Y)\times \{ 0 \}^{\dim F}$ and $\okbd_{X_\bullet}(R+\eps A) \supseteq \{ 0 \}^{\dim Y}\times \okbd_{F_\bullet}(W_\bullet^{\eps})$.
We thus have
$$\okbd_{Y_\bullet}(D_Y)\times \{0\}^{\dim F} + \{0\}^{\dim Y}\times\okbd_{F_\bullet}(W_\bullet^{\eps})=\okbd_{Y_\bullet}(D_Y) \times \okbd_{F_\bullet}(W_\bullet^{\eps})
$$
and
$$
\okbd_{X_\bullet}(f^*D_Y)+\okbd_{X_\bullet}(R+\eps A)
\supseteq\okbd_{Y_\bullet}(D_Y) \times \okbd_{F_\bullet}(W_\bullet^{\eps}).
$$
Since $D +\eps A \sim_{\Q} f^*D_Y + (R+\eps A)$, it follows from Remark \ref{additivity} and the above inclusion that
$$
\okbd_{X_\bullet}(D + \eps A) \supseteq \okbd_{X_\bullet}(f^*D_Y)+\okbd_{X_\bullet}(R+\eps A) \supseteq  \okbd_{Y_\bullet}(D_Y) \times \okbd_{F_\bullet}(W_\bullet^{\eps})
$$
and hence, we have
$$
\okbd_{X_\bullet}(D) = \bigcap_{\eps > 0} \okbd_{X_\bullet}(D+\eps A)
\supseteq\okbd_{X_\bullet}(f^*D_Y)+\bigcap_{\eps>0}\okbd_{X_\bullet}(R+\eps A)
\supseteq  \okbd_{Y_\bullet}(D_Y) \times \bigcap_{\eps > 0} \okbd_{F_\bullet}(W_\bullet^{\eps}).
$$
This implies that
$$
\vol_{\R^{\dim X}} (\okbd_{X_\bullet}(D)) \geq \vol_{\R^{\dim Y}}( \okbd_{Y_\bullet}(D_Y)) \cdot \vol_{\R^{\dim F}}\left(\bigcap_{\eps > 0} \okbd_{F_\bullet}(W_\bullet^{\eps}) \right).
$$

Since $F$ is a general fiber of $f \colon X \to Y$, we may assume that $F \not\subseteq \bm(R)$ and $F \not\subseteq \bp(R+\eps A)$ for any $\eps > 0$.
By \cite[(2.7) in p.804]{lm-nobody}, we know that
$$
\vol_{\R^{\dim F}} (\okbd_{F_\bullet}(W_\bullet^{\eps})) =  \frac{1}{\dim F!}\vol_{X|F}(R+\eps A).
$$
It thus follows that
$$
\vol_{\R^{\dim F}}\left(\bigcap_{\eps > 0} \okbd_{F_\bullet}(W_\bullet^{\eps}) \right) = \frac{1}{\dim F!}\vol_{X|F}^+(R).
$$
The first assertion of Theorem \ref{main1} now follows from \cite[Theorem A]{lm-nobody}.
If $f_*\mathcal{O}_X(mR)$ is weakly positive for every sufficiently divisible integer $m>0$, then by Lemma \ref{weakpos}, we have
$\vol_{X|F}^+(R)=   \vol_F(R|_F)$.
Thus by replacing $\vol_{X|F}^+(R)$ with $\vol_F(R|_F)$, we obtain the second assertion of Theorem~\ref{main1}.
\end{proof}

By Theorem \ref{main1} and Viehweg's theorem \cite[Theorem III]{V}, we obtain Corollary \ref{kawamata}.

\begin{proof}[Proof of Corollary \ref{kawamata}]
By the adjunction formula, we have $K_{X/Y}|_F = K_F$, and by Theorem \ref{var}, $K_{X/Y}$ is effective.
Recall that the coherent sheaf $f_* \omega_{X/Y}^{\otimes m}$ is weakly positive for every integer $m>0$ by Viehweg~\cite[Theorem III]{V}. Therefore, the assertion follows from Theorem \ref{main1}.
\end{proof}

Finally, we give a proof of Theorem \ref{main2}.

\begin{proof}[Proof of Theorem \ref{main2}]
\noindent $(1)\Rightarrow(2)$:
As in the proof of Theorem \ref{main1}, we consider a fiber-type admissible flag $X_\bullet$ on $X$ associated to admissible flags $Y_\bullet$ and $F_\bullet$ on $Y$ and $F$, respectively.
Take an ample divisor $A$ on $X$, and for $\eps > 0$, let $W_\bullet^{\eps}:=W_\bullet(K_{X/Y}+\eps A|F)$. Note that $\okbd_{F_\bullet}(W_\bullet^{\eps}) \supseteq \okbd_{F_\bullet}(K_F)$ for any $\eps > 0$. Recall from \cite[Theorem III]{V} that $f_* \omega_{X/Y}^{\otimes m}$ is weakly positive for every integer $m>0$. We have seen in the proof of Theorem \ref{main1} that
$\vol_{\R^{\dim F}}\left(\bigcap_{\eps > 0} \okbd_{F_\bullet}(W_\bullet^{\eps}) \right) = \vol_{\R^{\dim F}}(\okbd_{F_\bullet}(K_F))$, which implies that
$$
\bigcap_{\eps > 0} \okbd_{F_\bullet}(W_\bullet^{\eps}) = \okbd_{F_\bullet}(K_F).
$$
It is also shown in the proof of Theorem \ref{main1} that
$$
\okbd_{X_\bullet}(K_X) \supseteq
\okbd_{Y_\bullet}(K_Y) \times \okbd_{F_\bullet}(K_F) .
$$
If this inclusion is strict, then we have the strict inequality
\begin{align*}
\vol_{\R^{\dim X}}(\okbd_{X_\bullet}(K_X))&  >\vol_{\R^{\dim X}(\okbd_{Y_\bullet}(K_Y) \times \okbd_{F_\bullet}(K_F))}\\
&=\vol_{\R^{\dim Y}}(\okbd_{Y_\bullet}(K_Y))\cdot\vol_{\R^{\dim F}}(\okbd_{F_\bullet}(K_F)).
\end{align*}
By \cite[Theorem A]{lm-nobody}, we obtain a contradiction with the condition (1).

\smallskip

\noindent $(2)\Rightarrow(3)$:
Suppose that  $f$ is not birationally isotrivial, i.e., $\var(f) > 0$.
Then by Theorem \ref{var},
$$
\kappa(K_{X/Y}) \geq \kappa(K_F) + \var(f) > \kappa(K_F)=\dim F.
$$
As in the proof of $(1)\Rightarrow(2)$, we consider again a fiber-type admissible flag $X_\bullet$ on $X$ associated to admissible flags $Y_\bullet$ and $F_\bullet$ on $Y$ and $F$, respectively. Recall from the proof of Theorem \ref{main1} that we have
$$
\okbd_{X_\bullet}(K_X)\supseteq
\okbd_{Y_\bullet}(K_Y) + \bigcap_{\eps > 0}\okbd_{X_\bullet}(K_{X/Y}+\eps A) \supseteq
\okbd_{Y_\bullet}(K_Y) \times \okbd_{F_\bullet}(K_F).
$$

We claim that
$$
\dim \bigcap_{\eps > 0}\okbd_{X_\bullet}(K_{X/Y}+\eps A) > \dim \okbd_{F_\bullet}(K_F)(=\dim F).
$$
Note that $\bigcap_{\eps > 0}\okbd_{X_\bullet}(K_{X/Y}+\eps A) \supseteq \{0\}^{\dim Y}\times\okbd_{F_\bullet}(K_F)$.
To derive a contradiction, suppose that $\dim \bigcap_{\eps > 0}\okbd_{X_\bullet}(K_{X/Y}+\eps A)  = \dim F$. Then we must have 
\[\bigcap_{\eps > 0}\okbd_{X_\bullet}(K_{X/Y}+\eps A)  \subseteq \{ 0 \}^{\dim Y} \times \R^{\dim F}_{\geq 0}.\] Since $\bigcap_{\eps > 0}\okbd_{X_\bullet}(K_{X/Y}+\eps A) \supseteq  \okbd_{X_\bullet}(K_{X/Y})$, it follows that
$$\okbd_{X_\bullet}(K_{X/Y}) \subseteq \{ 0 \}^{\dim Y} \times \R^{\dim F}_{\geq 0}.$$
By the definition of the Okounkov body, this means that the support of any effective divisor linearly equivalent to $mK_{X/Y}$ for any $m \geq 1$ does not contain $F$. Thus we see that the natural restriction map
$$
H^0(X, \mathcal{O}_X(mK_{X/Y})) \longrightarrow H^0(F, \mathcal{O}_F(mK_F))
$$
is injective for every $m \geq 1$.
Therefore, we have $\kappa(K_{X/Y}) \leq \dim F$, which contradicts the condition $\kappa(K_{X/Y}) > \dim F$. We have thus shown the claim.

Now, by the claim, there exists a point $\mathbf x\in \bigcap_{\eps > 0}\okbd_{X_\bullet}(K_{X/Y}+\eps A)$ in $\R^{\dim X}$ such that $\mathbf x=\mathbf u+\mathbf u'$ for some $\mathbf u\in \{0\}^{\dim Y}\times\mathbb R^{\dim F}$ and non-zero $\mathbf u'\in\mathbb R^{\dim Y}\times\{0\}^{\dim F}$.
It is easy to see that the translation $\okbd_{Y_\bullet}(K_Y)+\mathbf x$ is not entirely contained in $\okbd_{Y_\bullet}(K_Y) \times \okbd_{F_\bullet}(K_F)$ since $\okbd_{Y_\bullet}(K_Y)+\mathbf u'\not\subseteq\okbd_{Y_\bullet}(K_Y)$ in $\mathbb R^{\dim Y}\times\{0\}^{\dim F}$. Therefore,
$$
\okbd_{Y_\bullet}(K_Y) + \bigcap_{\eps > 0}\okbd_{X_\bullet}(K_{X/Y}+\eps A)  \supsetneq \okbd_{Y_\bullet}(K_Y) \times \okbd_{F_\bullet}(K_F).
$$
Thus we obtain
$$
\okbd_{X_\bullet}(K_X)\supsetneq\okbd_{Y_\bullet}(K_Y)\times\okbd_{F_\bullet}(K_F)
$$
contradicting the condition (2).
\smallskip

\noindent $(3) \Rightarrow (1)$:
Suppose that $f \colon X \to Y$ is a birationally isotrivial fibration and $F$ is a general fiber.
Then there exist a generically finite cover $\tau \colon Y' \to Y$ with $Y'$ smooth and a surjective morphism $f' \colon X' \to Y'$ with connected fibers between smooth projective varieties such that $X'$ is birational to $F\times Y'$ and there is a generically finite surjective morphism $\tau' \colon X' \to X$. We may assume that there is a birational morphism $g \colon X' \to Y' \times F$. Denote the natural projections by $p_1 \colon F\times Y' \to F$ and $p_2 \colon F\times Y' \to Y'$. We obtain the following commutative diagram
\[
\begin{split}
\xymatrix{
 & F \times Y' \ar[dl]^-{p_1} \ar[dr]_-{p_2} & X' \ar[l]_-g \ar[r]^-{\tau'} \ar[d]^-{f'} & X \ar[d]^-f \\
 F & & Y' \ar[r]_-{\tau} & Y~.
}
\end{split}
\]
We can  write
$$
K_{X'}=\tau'^*K_X + R_{\tau'}~~\text{ and }~~ K_{Y'}=\tau^*K_Y + R_{\tau}
$$
where $R_{\tau'}$ and $R_{\tau}$ are effective ramification divisors of $\tau'$ and $\tau$, respectively. Note that
$
f'^* R_{\tau} \leq R_{\tau'}.
$
On the other hand, we have
$$
K_{X'} =g^*K_{F \times Y'} + E=f'^*K_{Y'} + g^*p_1^* K_F  + E = f'^*\tau^* K_Y + f'^*R_{\tau} + g^*p_1^* K_F  + E
$$
for some effective $g$-exceptional divisor $E$.
Thus we see that
$$
\tau'^*f^*K_Y + \tau'^*K_{X/Y}=\tau'^* K_X \leq f'^*\tau^*K_Y + g^* p_1^*K_F + E,
$$
and it gives us $\tau'^*K_{X/Y} \leq g^*p_1^*K_F + E$.
Since $(\tau'^*K_{X/Y})|_F = (\tau'|_F)^*(K_{X/Y}|_F) \cong K_F$, it follows that $g^* p_1^*K_F\leq\tau'^*K_{X/Y}$. Hence we obtain
$$
f'^*\tau^*K_Y + g^* p_1^*K_F \leq \tau'^* K_X \leq f'^*\tau^*K_Y + g^* p_1^*K_F + E.
$$
Since $f'^*\tau^*K_Y + g^* p_1^*K_F = g^*(p_1^*K_F + p_2^*\tau^*K_Y)$ and $E$ is $g$-exceptional, it follows that
$$
\frac{\vol_{X'}(\tau'^*K_X)}{\dim X!} = \frac{\vol_{Y'} (\tau^*K_Y)}{\dim Y!} \cdot \frac{\vol_F(K_F)}{\dim F!}.
$$
On the other hand,
$$
\vol_{X'}(\tau'^*K_X)  = \deg \tau' \cdot \vol_X(K_X)~~\text{ and }~~\vol_{Y'}(\tau^*K_Y)=\deg \tau \cdot \vol_Y(K_Y).
$$
Since $\deg \tau' = \deg \tau$, we finally obtain the desired condition (1).
This completes the proof.
\end{proof}

\begin{remark}
Our proof of $(2) \Rightarrow (3)$ in Theorem \ref{main2} also yields that if $f \colon X \to Y$ is an algebraic fiber space and both the base space $Y$ and a general fiber $F$ are of general type, then
$$
\text{$f$ is birationally isotrivial}~~ \Longleftrightarrow ~~\kappa(K_{X/Y})=\dim F.
$$
\end{remark}

\section{Examples}\label{exsec}

In this section, we present some examples. First, we show that Corollary \ref{kawamata} does not hold for the log setting under the assumptions of \cite[Conjecture 1.2]{Fujino}.

\begin{example}\label{ex0}
Let $X=\P(\mathcal{O}_{\P^1} \otimes \mathcal{O}_{\P^1}(-2))$ be the Hirzebruch surface with $(-2)$-curve $T$, and let $f \colon X \to C=\P^1$ be the canonical projection. Denote by $F$ a general fiber of $f$. Note that $-K_X \sim 2T+4F$ is nef and big.
Take an effective divisor $\Delta_C \sim_{\Q} -2K_C$ on $C$ with simple normal crossing support such that $(C, 2\Delta_C)$ is a klt pair. Note that $K_C + 2\Delta_C \sim_{\Q} -3K_C$ is ample. Since $2T + 4F$ is semiample, we can take a general irreducible smooth member $D_m \in |m(2T+4F)|$ for $m \gg 0$ such that $\Delta_X:=\frac{3}{2m}D_m + f^*\Delta_C$ is a simple normal crossing boundary divisor on $X$ and $(X, \Delta_X)$ is a klt pair.
One can easily check that
$$
K_X + \Delta_X \sim_{\Q} f^*(K_C + 2\Delta_C) + T.
$$
Furthermore, $K_X + \Delta_X \sim T+6F$ is ample and $(K_X + \Delta_X)|_F \sim_{\Q} T|_F$ is linearly equivalent to one point on $F\simeq \P^1$. Since $\text{Supp}(f^*2\Delta_C) \subseteq \text{Supp}(\Delta_X)$, the algebraic fiber space $f\colon X \to C$ and log pairs $(X, \Delta_X), (C, 2\Delta_C)$ satisfy the conditions in \cite[Conjecture 1.2]{Fujino}.
However, we have
$$
\vol_X(K_X + \Delta_X) =10, ~ \vol_{C}(K_C + 2\Delta_C)=6, \text{ and } \vol_F(K_F + \Delta_X|_F)=1
$$
so that we obtain
$$
\frac{1}{2}\vol_X(K_X + \Delta_X) < \vol_{C}(K_C + 2\Delta_C) \cdot \vol_F(K_F + \Delta_X|_F).
$$
\end{example}

\begin{remark}\label{Iitaka}
In Example \ref{ex0}, we also see that $f_*(\mathcal{O}_X(m((K_X+\Delta_X)-f^*(K_C + 2 \Delta_C))))$ is not weakly positive for any $m >0$.
On the other hand, by \cite[Theorem 4.13]{C}, if $f \colon X \to Y$ is an algebraic fiber space and $\Delta_X$ is a simple normal crossing effective divisor on $X$, then $f_*(\mathcal{O}_X(m(K_X + \Delta_X)-f^*K_Y))$ is weakly positive for any sufficiently divisible integer $m > 0$. In this setting, by Theorem \ref{main1}, we have
$$
\frac{\vol_X(K_X+\Delta_X)}{\dim X !} \geq \frac{\vol_Y (K_Y)}{\dim Y !} \cdot \frac{\vol_{F}(K_F+\Delta_X|_F)}{\dim F!}
$$
provided that $K_Y$ and $K_F+\Delta_X|_F$ are big.
Note that if $f \colon X \to Y$ is birationally isotrivial and $\Delta_X$ is ample, then this inequality can be strict (cf. Theorem \ref{main2}).
\end{remark}

The next two examples show that the equality in Theorem \ref{main1} is not equivalent to the birational isotriviality of the algebraic fiber space $f \colon X \to Y$ in general. In particular, we also see that Theorem \ref{main2} does not hold in the log setting.

\begin{example}\label{ex1}
Let $E$ be a rank $2$ vector bundle on a smooth projective curve $C$, and $X:=\P(E)$ the corresponding ruled surface with a tautological divisor $T$ on $X$ so that $\mathcal{O}_X(T)=\mathcal{O}_{\P(E)}(1)$.
Let $f \colon X \to C$ be the canonical projection, and $F$ a general fiber of $f$. Note that $f$ is an isotrivial fibration.
By tensoring $E$ with a suitable ample divisor, we may assume that the coherent sheaf $f_* \mathcal{O}_X(mT)$ is weakly positive for every integer $m > 0$ and $T^2 > 0$.
We can choose $a, b \in \Z_{>0}$ such that the divisor $D=aF + bT$ on $X$ is ample. We can write $aF = f^*D_C$ for some effective divisor $D_C$ on $C$ so that
$D=f^*D_C + bT$.
Then we have
$$
\vol_X(D)=2ab + b^2 T^2, ~~\vol_C(D_C)=a,~~\vol_F(T|_F) = b.
$$
This implies that
$$
\frac{\vol_X(D)}{2} > \vol_C(D_C) \cdot \vol_F(T|_F)
$$
even though $f$ is isotrivial.

Note that we can find $a, b \in \Z_{>0}$ such that $D - K_X$, $D_C - K_C$, and $bT|_F - K_F$ are ample simultaneously. Thus there are effective divisors $\Delta \sim_{\Q} D - K_X, \Delta_C \sim_{\Q} D_C - K_C$, and $\Delta_F \sim_{\Q} bT|_F - K_F$ such that $(X, \Delta), (C, \Delta_C)$, and $(F, \Delta_F)$ are klt pairs and $K_X+\Delta$, $K_C + \Delta_C$, $K_F + \Delta_F$ are big divisors. Even though
$K_X + \Delta \sim_{\Q} f^*(K_C + \Delta_C) + bT$, we have
\[
\frac{\vol_X(K_X + \Delta)}{2} > \vol_C(K_C + \Delta_C) \cdot \vol_F(K_F + \Delta_F).
\]
\end{example}

\begin{example}\label{ex2}
Let $X$ be a K3 surface admitting two distinct elliptic fibrations $f_i \colon X \to \P^1$ with $F_i$ a general fiber for $i=1 \ \text{and}\ 2$. We have $F_1 \cdot F_2 > 0$. We assume that $f_1$ is not birationally isotrivial.
It is easy to see that the coherent sheaf $f_{1 *}\mathcal{O}_X(mF_2)$ is weakly positive for every integer $m > 0$.
Consider a divisor $D=aF_1 + b F_2$ on $X$ for $a, b \in \Q_{>0}$. We can write $aF_1 = f_1^* D_1$ for some divisor $D_1$ on $\P^1$ so that
$D=f_1^*D_1 + bF_2$.
Then we have
$$
\vol_X(D)=2ab(F_1 \cdot F_2), ~~ \vol_{\P^1}(D_1)=a, \text{ and } \vol_{F_1}(bF_2|_{F_1}) = b (F_1 \cdot F_2).
$$
This shows that
$$
\frac{\vol_X(D)}{2} = \vol_{\P^1}(D_1) \cdot \vol_{F_1}(bF_2|_{F_1})
$$
even though $f_1$ is not birationally isotrivial.

As in the previous example, we can choose $a, b \in \Z_{>0}$ such that $D=D-K_X$, $D_1 - K_{\P^1}$, and $bF_2|_{F_1} - K_{F_1}$ are ample simultaneously. Thus there are effective divisors $\Delta\sim_{\Q} D$, $\Delta_1 \sim_{\Q} D_1 - K_{\P^1}$ and $\Delta_2 \sim_{\Q} bF_2|_{F_1} - K_{F_1}$ such that $(X,\Delta)$, $(\P^1, \Delta_1)$, and $(F_1, \Delta_2)$ are klt pairs and $K_X + \Delta$, $K_{\P^1}+\Delta_1$, and $K_{F_1}+\Delta_2$ are big divisors. Even though
$K_X + D \sim_{\Q} f_1^*(K_{\P^1} + \Delta_1) + b F_2$,
the following equality holds:
$$
\frac{\vol_X(K_X + D)}{2} = \vol_{\P^1}(K_{\P^1}+\Delta_1) \cdot \vol_{F_1}(K_{F_1} + \Delta_2).
$$
\end{example}

\end{document}